\documentclass[11pt]{article}
\usepackage{amsmath, amsfonts, amsthm, amssymb, color}
\usepackage{graphicx}
\usepackage{float}
\usepackage{verbatim}

\allowdisplaybreaks

\numberwithin{equation}{section}

\newtheorem{lemma}{Lemma}[section]

\newtheorem{prop}[lemma]{Proposition}
\newtheorem{thm}[lemma]{Theorem}
\newtheorem{cor}[lemma]{Corollary}
\theoremstyle{definition}

\newtheorem{example}[lemma]{Example}

\theoremstyle{remark}

\def\R{\mathbb{R}}
\def\N{\mathbb{N}}

\def\a{\mathbf{a}}
\def\b{\mathbf{b}}
\def\i{\mathbf{i}}

\def\p{\mathbf{p}}

\def\A{\mathcal{A}}

\def\x{\mathbf{x}}
\def\p{\mathbf{p}}
\def\q{\mathbf{q}}
\def\P{\mathbb{P}}
\def\supp{\mathrm{supp}}

\numberwithin{equation}{section} \numberwithin{table}{section}

\title{Disintegration results for fractal measures and applications to Diophantine approximation}
\author{Simon Baker\\\\
	\emph{Department of Mathematical Sciences,} \\ \emph{Loughborough University,} \\ \emph{Loughborough, LE11 3TU, UK} \\ Email: simonbaker412@gmail.com\\ \\
}

\date{\today}
\begin{document}
	\maketitle

	\begin{abstract}
In this paper we prove disintegration results for self-conformal measures and affinely irreducible self-similar measures. The measures appearing in the disintegration resemble self-conformal/self-similar measures for iterated function systems satisfying the strong separation condition. As an application of our results, we prove the following Diophantine statements:
\begin{enumerate}
	\item Using a result of Pollington and Velani, we show that if $\mu$ is a self-conformal measure in $\R$ or an affinely irreducible self-similar measure, then there exists $\alpha>0$ such that for all $\beta>\alpha$ we have $$\mu\left(\left\{\x\in \R^{d}:\max_{1\leq i\leq d}|x_{i}-p_i/q|\leq \frac{1}{q^{\frac{d+1}{d}}(\log q)^{\beta}}\textrm{ for i.m. }(p_1,\ldots,p_d,q)\in \mathbb{Z}^{d}\times \N\right\}\right)=0.$$ 
	\item Using a result of Kleinbock and Weiss, we show that if $\mu$ is an affinely irreducible self-similar measure, then $\mu$ almost every $\x$ is not a singular vector. 
\end{enumerate}
		
	\end{abstract}
	

	\section{Introduction}
One of the most challenging problems in Fractal Geometry is to understand how a stationary measure distributes mass when the underlying iterated function system is overlapping. The exact overlaps conjecture and the study of Bernoulli convolutions are two particular instances of this problem (see \cite{Hochman2,Hochman,Rap,RapVar,Shmerkin,Var2,Var3} and the references therein). In this paper, we show that for self-conformal and self-similar measures, it is possible to disintegrate these measures over a family of measures for which we have lots of control over how mass is distributed. Our main results are the following two statements.

\begin{thm}
	\label{thm:self-conformal}
	Let $\mu$ be a non-atomic self-conformal measure on $\R^{d}$. Then there exists a probability space $(\Omega,\mathcal{A},\P)$ and a family of measures $\{\mu_{\omega}\}_{\omega\in \Omega}$ such that the following properties are satisfied:
	\begin{enumerate}
		\item $\mu=\int \mu_{\omega}\, d\P(\omega)$
		\item There exists $C_{1}>0$ such that for any $\omega\in \Omega,$ $x\in \mathrm{supp}(\mu_{\omega})$ and $r>0$ we have $\mu_{\omega}(B(x,2r))\leq C_{1}\mu_{\omega}(B(x,r)).$
		\item There exists $C_{2},\alpha>0$ such that for any $\omega\in \Omega$, $x\in \supp(\mu_{\omega}),$ $y\in\mathbb{R}^{d},$ $0<r\leq 1$ and $\epsilon>0$ we have 
		$\mu_{\omega}(B(y,\epsilon r)\cap B(x,r))\leq C_{2}\epsilon^{\alpha}\mu_{\omega}(B(x,r)).$ 
	\end{enumerate}
\end{thm}
	
\begin{thm}
	\label{thm:self-similar}
	Let $\mu$ be an affinely irreducible self-similar measure on $\R^{d}$. Then there exists a probability space $(\Omega,\mathcal{A},\P)$ and a family of measures $\{\mu_{\omega}\}_{\omega\in \Omega}$ such that the following properties are satisfied:
	\begin{enumerate}
		\item $\mu=\int \mu_{\omega}\, d\P(\omega)$
		\item There exists $C_{1}>0$ such that for any $\omega\in \Omega,$ $x\in \supp(\mu_{\omega})$ and $r>0$ we have $\mu_{\omega}(B(x,2r))\leq C_{1}\mu_{\omega}(B(x,r)).$
		\item There exists $C_{2},\alpha>0$ such that for any $\omega\in \Omega$, $x\in \supp(\mu_{\omega}),$ affine subspace $W<\mathbb{R}^{d},$ $0<r\leq 1$ and $\epsilon>0$ we have 
		$\mu_{\omega}(W^{(\epsilon r)}\cap B(x,r))\leq C_{2}\epsilon^{\alpha}\mu_{\omega}(B(x,r)).$ 
	\end{enumerate}
\end{thm} Precise definitions for each of the objects appearing in these theorems are given in Section \ref{sec:preliminaries}. We merely state at this point that given an affine subspace $W<\R^{d}$ and $\epsilon>0,$ we let $W^{(\epsilon)}$ denote the $\epsilon$ neighbourhood of $W$. The $\mu_{\omega}$ appearing in Theorems \ref{thm:self-conformal} and \ref{thm:self-similar} also satisfy a form of dynamical self-similarity reminiscent of that satisfied by stationary measures (see \eqref{eq:dynamical self-similarity}). In our applications we will only use Theorem \ref{thm:self-conformal} in the case when $d=1$. In this case the third property coincides with third property in Theorem \ref{thm:self-similar}.

The measures $\mu_{\omega}$ appearing in Theorems \ref{thm:self-conformal} and \ref{thm:self-similar} have many of the properties we would observe for a self-conformal or self-similar measure when the underlying iterated function system satisfies the strong separation condition. Thus these theorems provides a connecting bridge between the overlapping case and the simpler case when the underlying iterated function system satisfies the strong separation condition. Indeed if a property is known to hold for $\mu$-almost every $x$ under the assumption that the underlying IFS satisfies the strong separation condition, then it reasonable to expect that Theorems \ref{thm:self-conformal} and \ref{thm:self-similar} would provide a route to showing that a property hold for $\mu$-almost every $x$ without this assumption. We will show how this strategy can be implemented in the next section to prove results on the Diophantine properties of self-conformal and self-similar measures. 


We emphasise that if the underlying iterated function system also satisfies the strong separation condition, then it is reasonably straightforward to show that the conclusions of Theorems \ref{thm:self-conformal} and \ref{thm:self-similar} hold. In this case, we can simply take $\Omega$ to be a single element set and define $\mu=\mu_{\omega}$ for all $\omega\in \Omega$. The same cannot be said if the iterated function system only satisfies the open set condition. In Section \ref{sec:example} we include an explicit example demonstrating this fact. 

\subsection{Diophantine approximation}
One of the main motivations behind Theorems \ref{thm:self-conformal} and \ref{thm:self-similar} comes from applications to Diophantine approximation. In this section we detail two such applications. 
\subsubsection{$\Psi$-well approximable vectors} At its core, Diophantine approximation is concerned with approximating elements of $\mathbb{R}^{d}$ with rational vectors. A standard framework for this problem is as follows. Given a decreasing function $\Psi:\mathbb{N}\to (0,\infty)$ let $$W(\Psi):=\left\{\x\in \mathbb{R}^{d}:\max_{1\leq i\leq d}|x_i-p_i/q|\leq \Psi(q)\textrm{ for i.m. } (p,q)\in\mathbb{Z}^{d}\times \N\right\}.$$ We will always assume that $\Psi$ is decreasing. In the special case where $\Psi(n)=n^{-\tau}$ we denote $W(\Psi)$ by $W(\tau)$. A well known result of Dirichlet implies that $W(\frac{d+1}{d})=\mathbb{R}^{d}$. Moreover, it is a simple consequence of the Borel-Cantelli lemma that $W(\tau)$ has zero Lebesgue measure whenever $\tau>\frac{d+1}{d}$. Much more is known about the Lebesgue measure of the sets $W(\Psi)$. A well known theorem due to Khintchine shows that the Lebesgue measure of $W(\Psi)$ is determined by naturally occurring volume sums \cite{Khi}. For more on the Lebesgue theory we refer the reader to \cite{ABH,BBDV,Har,KouMay} and the references therein. Once a result has been established for the Lebesgue measure, it is a natural and well studied problem to show that the analogous statement holds for a fractal measure. It is often reasonable to expect that what has been observed for the Lebesgue measure should persist for a fractal measure. This problem has attracted significant attention recently (see \cite{ABCY,ACY,Bakapprox,Bak,BFR,Bug2,BugDur,DFSU,DSFU2,DJ,FMS,KL,KLW,LSV,Mah,SW,TWW,Yu} and the references therein). We do not attempt to give an exhaustive overview of this topic but instead recap a few results to properly contextualise our work. A measure is said to be extremal if it gives zero measure to $W(\tau)$ for any $\tau>\frac{d+1}{d}$. Much of the work on this topic was initially concerned with proving that fractal measures are extremal. In \cite{KLW} Kleinbock, Lindenstrauss, and Weiss introduced the notion of a friendly measure and showed that any friendly measure is extremal. We do not give the definition of a friendly measure, but merely state that the natural Hausdorff measure restricted to a self-similar set coming from an iterated function system satisfying the open set condition is friendly provided the self-similar set is not contained in an affine subspace. We emphasise that the natural Hausdorff measure supported on such a self-similar set is equivalent to a particular choice of self-similar measure. Thus this statement has an equivalent version in the language of self-similar measures. In \cite{DFSU,DSFU2} Das et al introduced and studied a property of measures they referred to as quasi-decaying. In \cite{DFSU} they proved that quasi-decaying measures are extremal. In \cite{DSFU2} they showed that many natural dynamically interesting measures are quasi-decaying and are therefore extremal. These include self-similar measures and self-conformal measures when the underlying iterated function system satisfying a weak irreducibility assumption. Crucially the results of \cite{DSFU2} do not require any separation assumptions on the underlying iterated function system. More recently, some significant progress has been made in the context of self-similar measures. Khalil and Luethi proved that if a self-similar IFS is defined by rational parameters and satisfies the open set condition, then an analogue of Khintchine's theorem holds for any of its self-similar measures satisfying an additional dimension assumption \cite{KL}. This result was recently built upon by B\'{e}nard, He, and Zhang \cite{BHZ} who proved that for any self-similar measure in $\R$ an analogue of Khintchine's theorem holds. We finish this discussion by giving a more detailed account of a result due to Pollington and Velani \cite{PolVel} which is particularly important to us. They proved the following statement.

\begin{thm}
\label{thm:Pollington and Velani}
Let $\mu$ be a compactly supported measure on $\mathbb{R}^{d}$. Suppose that there exists $C_1,C_2,\alpha>0$ such that the following properties are satisfied:
\begin{enumerate}
	\item For any $x\in \supp(\mu)$ and $r>0$ we have $\mu(B(x,2r))\leq C_{1}\mu(B(x,r))$.
	\item For any $x\in \supp(\mu),$ affine subspace $W<\R^{d}$, $0\leq r \leq 1$ and $\epsilon>0$ we have $\mu(W^{(\epsilon r)}\cap B(x,r))\leq C_{2}\epsilon^{\alpha}\mu(B(x,r))$.
\end{enumerate}
Then $\mu(W(\Psi))=0$ for any $\Psi$ satisfying $$\sum_{n=1}^{\infty}n^{\alpha\frac{d+1}{d} -1}\Psi(n)^{\alpha}<\infty.$$
\end{thm} 
Theorem \ref{thm:Pollington and Velani} is a generalisation of a result due to Weiss in $\R$ \cite{Weiss}. For our purposes, the significance of this theorem is that it allows us to make a stronger conclusion than extremality. In particular, the following statement immediately follows.

\begin{cor}
	\label{cor:Pollington and Velani}
	Let $\mu$ satisfy the assumptions of Theorem \ref{thm:Pollington and Velani} and let $\alpha$ be as in the statement of this theorem. Then for any $\beta>1/\alpha$ we have $\mu(W(\Psi))=0$ for $\Psi:\N\to (0,\infty)$ given by $\Psi(n)= n^{-\frac{d+1}{d}}(\log n)^{-\beta}$.
\end{cor}
Pollington and Velani then applied Theorem \ref{thm:Pollington and Velani} to show that its conclusion holds for the natural Hausdorff measure on a self-similar set when the open set condition is satisfied and the underlying iterated function system is affinely irreducible. The only genuine obstacle to the conclusion of Theorem \ref{thm:Pollington and Velani} is for the measure to be supported on an affine subspace. In that case it can be shown that the conclusion of this theorem may be false. The conclusion of Theorem \ref{thm:Pollington and Velani} should hold more generally for self-conformal and self-similar measures with no separation assumptions.  The following result which follows from Theorems \ref{thm:self-conformal}, \ref{thm:self-similar} and \ref{thm:Pollington and Velani} shows that this is indeed the case. 

\begin{thm}
	The following statements are true:
	\begin{enumerate}
		\item Let $\mu$ be a self-conformal measure in $\R$. Then there exists $\alpha>0$ such that $\mu(W(\Psi))=0$ for any $\Psi:\N\to (0,\infty)$ satisfying $$\sum_{n=1}^{\infty}n^{2\alpha-1}\Psi(n)^{\alpha}<\infty.$$ 
		\item Let $\mu$ be an affinely irreducible self-similar measure in $\R^{d}$. Then there exists $\alpha>0$ such that $\mu(W(\Psi))=0$ for any $\Psi:\N\to (0,\infty)$ satisfying $$\sum_{n=1}^{\infty}n^{\alpha\frac{d+1}{d} -1}\Psi(n)^{\alpha}<\infty.$$ 
	\end{enumerate}
\end{thm}  
\begin{proof}
	Apply Theorem \ref{thm:Pollington and Velani} to each $\mu_{\omega}$ appearing in Theorems \ref{thm:self-conformal} and \ref{thm:self-similar}, then use the fact $\mu=\int \mu_{\omega}\, d\mathbb{P}(\omega)$.
\end{proof}

\subsubsection{Singular vectors}
A vector $(x_1,\ldots,x_d)\in\R^d$ is said to be singular if for every $c>0$ the system of inequalities $$\max_{1\leq i\leq d}\|qx_i\|\leq ct^{-1/d},\quad \, 0<q<t,$$ has an integer solution for all $t>0$ sufficiently large. Here we let $\|\cdot\|$ denote the distance to the nearest integer. We let $\textbf{Sing}_{d}$ denote the set of singular vectors in $\R^{d}$. It was shown in \cite{DavSch} that $\textbf{Sing}_{d}$ has zero Lebesgue measure. An argument due to Khintchine shows that in $\R$ we have $\textbf{Sing}_{1}=\mathbb{Q}$ \cite{Khi2}. More generally, in higher dimensions it is known that $\textbf{Sing}_{d}$ contains every rational hyperplane in $\R^{d}$. Recently, Cheung and Chevallier \cite{CheChe} showed that $\dim_{H}(\textbf{Sing}_{d})=\frac{d^{2}}{d+1}$ for all $d\geq 2$. This built on earlier work of Cheung who proved this result in the special case when $d=2$ \cite{Che}. Since $\frac{d^2}{d+1}\geq d-1$ this result implies that $\textbf{Sing}_{d}$ does not consist solely of every rational hyperplane in $\R^{d}$. For more on singular vectors we refer the reader to \cite{BGMRV,Kha,KMW}. For our purposes we will need the following result due to Kleinbock and Weiss \cite{KW}.

\begin{thm}
	\label{thm:Kleinbock and Weiss}
	Let $\mu$ be a friendly measure on $\mathbb{R}^{d}$. Then $\mu(\textbf{Sing}_{d})=0$
\end{thm}Kleinbock and Weiss used this result to show that the natural Hausdorff measure restricted to a self-similar set coming from an iterated function system satisfying the open set condition gives zero mass to $\textbf{Sing}_{d}$ provided the self-similar set is not contained in an affine subspace. We emphasise that we have not defined what it means for a measure to be friendly. For a definition we refer the reader to \cite{KLW}. At this point we merely remark that if a measure satisfies the assumptions of Theorem \ref{thm:Pollington and Velani} then it is friendly. The conclusion $\mu(\textbf{Sing}_{d})=0$ should hold for a more general family of self-similar measures. The only genuine obstruction is for the measure to be supported on an affine subspace. The following statement follows from Theorems \ref{thm:self-similar} and \ref{thm:Kleinbock and Weiss} and shows that this is the case.

\begin{thm}
	Let $\mu$ be an affinely irreducible self-similar measure on $\mathbb{R}^{d}$. Then $\mu(\textbf{Sing}_{d})=0$.
\end{thm} 
\begin{proof}
Let $\{\mu_{\omega}\}_{\omega\in \Omega}$ be as in Theorem \ref{thm:self-similar}. The second and third properties of this theorem imply that each $\mu_{\omega}$ is a friendly measure. Now apply Theorem \ref{thm:Kleinbock and Weiss} to each $\mu_{\omega}$ and use the fact $\mu=\int \mu_{\omega}\, d\P(\omega)$ to conclude our result.
\end{proof}
Theorem \ref{thm:self-conformal} and Theorem \ref{thm:Kleinbock and Weiss} imply an analogous statement for self-conformal measures when $d=1$. However, by the aforementioned result of Khintchine in this case $\textbf{Sing}_{1}=\mathbb{Q}$ so this result holds trivially. 

\section{Preliminaries}
\label{sec:preliminaries}
In this section we will introduce some notation, collect some useful results from Fractal Geometry, and provide a general framework for our disintegration technique.

\subsection{Notation}
Given a set $S$ and $f,g:S\to (0,\infty)$ we write $f\ll g$ if there exists $C>0$ such that $f(s)\leq Cg(s)$ for all $s\in S$. We write $f\asymp g$ if $f\ll g$ and $g\ll f$. Given an alphabet $\A$ we let $\A^*=\cup_{n=1}^{\infty}\A^{n}$ denote the set of finite words with entries in $\A$. Given two finite words $\a,\b\in \A^{*}$ we let $\a\wedge \b$ denote the maximal common prefix of $\a$ and $\b$. If no such prefix exists we define $\a\wedge \b$ to be the empty word. We denote the length of a finite word $\a\in \A^*$ by $|\a|.$ Given $a\in \A$ and $n\in \N$ we let $a^{n}=\overbrace{a\cdots a}^{\times n}.$
\subsection{Fractal Geometry}
Let $Y$ be a closed set. A map $\varphi:Y\to Y$ is called a contraction if there exists $r\in (0,1)$ such that $\|\varphi(x)-\varphi(y)\|\leq r\|x-y\|$ for all $x,y\in Y$. An iterated function system (IFS) on $Y$ is a finite set of contractions acting on $Y$. We will often suppress $Y$ from our discussion and simply speak of an iterated function system or IFS. A well known result due to Hutchinson \cite{Hut} states that for any IFS $\Phi=\{\varphi_{a}\}_{a\in \A}$ there exists unique non-empty compact set $X\subset\mathbb{R}^{d}$ satisfying $$X=\bigcup_{a\in \A}\varphi_{a}(X).$$ We call $X$ the invariant set of the IFS. We will always assume that two of the contractions in our IFS have distinct fixed points. This ensures the invariant set is non-trivial. When an IFS $\Phi$ consists of similarities we say that it is a self-similar IFS and the invariant set is a self-similar set. In the special case where each contraction in our IFS is $C^{1}$ and angle preserving we will say that the IFS is self-conformal. That is $\Phi=\{\varphi_a\}_{a\in \A}$ is self-conformal if for all $a\in \A$, $x\in Y,$ and $y\in \mathbb{R}^{d}$ we have $\|\varphi_{a}'(x)y\|=\|\varphi_{a}'(x)\|\cdot \|y\|$. We will always assume that the contractions in a self-conformal IFS satisfy the following additional H\"{o}lder regularity assumption: There exists $\alpha>0$ such that $$\big|\|\varphi_{a}'(x)\|-\|\varphi_{a}'(y)\|\big|\ll \|x-y\|^{\alpha}$$ for any $x,y\in Y$ and $a\in \A$. We will say that a self-similar IFS is affinely irreducible if there does not exist an affine subspace in $\R^{d}$ that is preserved by each element of the IFS. This is equivalent to the property that the invariant set is not contained in an affine subspace of $\R^d$.  Given $\a=(a_1,\ldots,a_n)\in \A^{*}$ we let $\varphi_{\a}=\varphi_{a_1}\circ \cdots\circ  \varphi_{a_n}$.

Given an IFS $\Phi=\{\varphi_{a}\}_{a\in \A}$ and a probability vector\footnote{$\sum_{a\in \A} p_a=1$ and $p_{a}>0$ for all $a\in \A$.} $\p=(p_{a})_{a\in \A}$  there exists a unique Borel probability measure $\mu$ satisfying $$\mu=\sum_{a\in \A}p_{a}\varphi_{a}\mu$$ where $\varphi_{a}\mu$ is the pushforward of $\mu$ under $\varphi_{a}$. We call $\mu$ the stationary measure corresponding to $\Phi$ and $\p$. It is a consequence of our underlying assumption that the invariant set is non-trivial that $\mu$ is always non-atomic. When $\Phi$ consists of similarities we will say that $\mu$ is the self-similar measure corresponding to $\Phi$ and $\p$. Similarly, when $\Phi$ is a self-conformal IFS we will say that $\mu$ is the self-conformal measure corresponding to $\Phi$ and $\p$. We will say that a self-similar measure is affinely irreducible if the underlying self-similar IFS is affinely irreducible.

We finish this section by stating the following lemma that records some useful properties of self-conformal IFSs. For a proof of this lemma see \cite[Lemma 6.1]{AKT}.
\begin{lemma}
	\label{Lem:Sascha lemma}
	Let $\Phi$ be a self-conformal IFS. Then for any $x,y\in Y$ and $\a\in \A^{*}$ we have $$\|\varphi_{\a}(x)-\varphi_{\a}(y)\|\asymp Diam (\varphi_{\a}(X))\cdot \|x-y\|.$$ Moreover, for any $\a,\b\in \A^{*}$ we have $$Diam(\varphi_{\a\b}(X))\asymp Diam(\varphi_{\a}(X))Diam(\varphi_{\b}(X)).$$
\end{lemma}

\subsection{A general disintegration framework}
In this section we will provide a general framework for disintegrating stationary measures. The ideas appearing in this section have their origins in a paper of Galicer et al \cite{GSSY}. These ideas have been used to study the absolute continuity of self-similar measures \cite{KaeOrp,SSS}, normal numbers in fractal sets \cite{ABS}, and for studying problems related to the Fourier decay of stationary measures \cite{AHW3,AHW4,BakBan,BKS,BS}.

Suppose we are given an IFS $\Phi=\{\varphi_{a}\}_{a\in \A}$ and a probability vector $\p$. Let $\A_{1},\ldots, \A_{k}\subset \A$ be a collection of non-empty subsets of $\A$ satisfying $\cup_{i=1}^{k}\A_{i}=\A$. We emphasise that we do not assume that $\A_{i}\cap \A_{j}=\emptyset$ for $i\neq j$. We let $I=\{1,\ldots, k\}$ and $\Omega=I^{\N}$. We let $\sigma:\Omega\to \Omega$ denote the usual left shift map, i.e. $\sigma((i_n))=(i_{n+1})$ for all $(i_n)\in \Omega$. We define a probability vector $\q=(q_i)_{i=1}^{k}$ according to the rule $$q_{i}=\sum_{a\in \A_{i}}\frac{p_{a}}{\#\{j\in I: a\in \A_{j}\}}$$ for each $i\in I$. We let $\P$ denote the infinite product measure on $\Omega$ corresponding to $\q$. Note that $\P$ is $\sigma$-invariant. For each $i\in I$ we define a probability vector $\q_{i}=(q_{a}^{i})_{a\in \A_{i}}$ according to the rule $$q_{a}^{i}=\frac{1}{q_{i}}\frac{p_{a}}{\#\{j\in I: a\in \A_{j}\}}$$ for each $i\in \A_{i}$. Given $\omega=(i_n)\in \Omega$ we let $\Sigma_{\omega}=\prod_{n=1}^{\infty}\A_{i_n}$ and $m_{\omega}=\prod_{n=1}^{\infty}\q_{i_n}.$ We emphasise that for all $\omega\in \Omega$ the support of $m_{\omega}$ is $\Sigma_{\omega}$. Given $\omega\in \Omega$ we also let $\Pi_{\omega}:\Sigma_{\omega}\to\mathbb{R}^{d}$ be given $$\Pi_{\i}((a_n)_{n=1}^{\infty})= \lim_{n\to\infty}\varphi_{a_1\ldots a_n}(\mathbf{x}).$$ Here $\mathbf{x}$ is any vector in the domain of our IFS. For each $\omega\in \Omega$ we let $X_{\omega}=\Pi_{\omega}(\Sigma_{\omega})$. It follows from the definition that $X_{\omega}$ satisfies a form of dynamical self-similarity that resembles that satisfied by invariant sets:
\begin{equation}
\label{eq:set dynamical self-similarity}X_{\omega}=\bigcup_{a\in \A_{i_1}}\varphi_{a}(X_{\sigma(\omega)}).
\end{equation} Moreover, iterating \eqref{eq:set dynamical self-similarity} we have the following relation for all $\omega\in \Omega$ and $n\in \N$:
$$X_{\omega}=\bigcup_{\a\in \prod_{j=1}^{n}\A_{i_j}}\varphi_{\a}(X_{\sigma^{n}\omega}).$$

Last of all, we define $\mu_{\omega}=\Pi_{\omega}m_{\omega}$. It follows immediately from the definition of $\mu_{\omega}$ that we have the following relation which resembles that satisfied by stationary measures:
\begin{equation}
	\label{eq:dynamical self-similarity}
	\mu_{\omega}=\sum_{a\in \A_{i_1}}q_{a}^{i_1}\varphi_{a}\mu_{\sigma(\omega)}.
\end{equation} Iterating \eqref{eq:dynamical self-similarity} we have the following relation for all $\omega\in \Omega$ and $n\in \N$:
$$\mu_{\omega}=\sum_{\a\in \prod_{j=1}^{n}\A_{i_j}}\prod_{j=1}^{n}q_{a_j}^{i_j}\cdot \varphi_{\a}\mu_{\sigma^{n}\omega}.$$

The following result shows that the above framework always gives a disintegration of a stationary measure.

\begin{prop}
	\label{prop:Disintegration prop}
	Let $\mu$ be the stationary measure for an IFS $\{\varphi_{a}\}_{a\in \A}$ and a probability vector $\p$. Let $\A_{1},\ldots,\A_{k}\subset \A$ satisfy $\cup_{i=1}^{k}\A_{i}=\A$. Then for the family of measures $\{\mu_{\omega}\}_{\omega\in \Omega}$ defined above, we have the following disintegration of $\mu$:
	$$\mu=\int \mu_{\omega}\, d\P(\omega).$$ 
\end{prop}
Various versions of Proposition \ref{prop:Disintegration prop} appear in the literature (see \cite{ABS,BakBan,GSSY,SSS} for instance). None of these results quite recover Proposition \ref{prop:Disintegration prop} so we include what is now the standard proof. 

\begin{proof}
Let $\tilde{\mu}=\int \mu_{\omega}\, d\P(\omega).$ We will show that $\tilde{\mu}$ satisfies 
\begin{equation}
	\label{eq:unique soln}
	\tilde{\mu}=\sum_{a\in \A}p_{a}\varphi_{a}\tilde{\mu}.
\end{equation}Since $\mu$ is the unique measure satisfying \eqref{eq:unique soln} it will follow that $\tilde{\mu}=\mu$ and our proof is complete. We now show that \eqref{eq:unique soln} holds:
\begin{align*}
	\tilde{\mu}=\int \mu_{\omega}\, d\P(\omega)&=\sum_{i=1}^{k}\int_{\omega:i_1=i}\mu_{\omega}\,d\P(\omega)\\
	&=\sum_{i=1}^{k}\int_{\omega:i_1=i}\sum_{a\in \A_{i}}q_{a}^{i}\varphi_{a}\mu_{\sigma(\omega)}\, d\P(\omega)\qquad (\textrm{By }\eqref{eq:dynamical self-similarity})\\
	&=\sum_{i=1}^{k}\sum_{a\in \A_{i}}q_{a}^{i}\int_{\omega:i_1=i} \varphi_{a}\mu_{\sigma(\omega)}\, d\P(\omega)\\
	&=\sum_{i=1}^{k}\sum_{a\in \A_{i}}q_{a}^{i}q_{i}\int \varphi_{a}\mu_{\omega}\, d\P(\omega)\qquad (\textrm{By }\sigma\textrm{-invariance of }\P\textrm{ and independence})\\
	&=\sum_{i=1}^{k}\sum_{a\in \A_{i}}\frac{p_{a}}{\#\{j\in I:a\in \A_j\}}\int \varphi_{a}\mu_{\omega}\, d\P(\omega)\\
	&=\sum_{a\in \A}p_{a}\int \varphi_{a}\mu_{\omega}\, d\P(\omega)\\
	&=\sum_{a\in \A}p_{a}\varphi_{a}\tilde{\mu}.
\end{align*}Thus \eqref{eq:unique soln} holds and our proof is complete.

\end{proof}

\section{Proofs of Theorems \ref{thm:self-conformal} and \ref{thm:self-similar}}
In this section we will prove Theorems \ref{thm:self-conformal} and \ref{thm:self-similar}. We present the proof of Theorem \ref{thm:self-conformal} first as it is the simplest and it will serve as a warm-up for the more demanding proof in Theorem \ref{thm:self-similar}.

\subsection{Proof of Theorem \ref{thm:self-conformal}}
Let us fix $\Phi$ a self-conformal IFS and $\p$ a probability vector. We let $X$ denote the invariant set of $\Phi$. Appealing to our underlying assumption that the IFS is non-trivial, we can assert that there exists $N\in \N$ and $a_1,a_2\in \A$ such that $$\varphi_{a_1^{N}}(X)\cap \varphi_{a_{2}^{N}}(X)=\emptyset.$$ Replacing $N$ with a potentially large integer, we can guarantee that the following property also holds: For any $\a\in \A^{N}$,  either $$\varphi_{\a}(X)\cap \varphi_{a_{1}^{N}}(X)=\emptyset$$ or 
$$\varphi_{\a}(X)\cap \varphi_{a_{2}^{N}}(X)=\emptyset.$$ It follows from the above, and the well know fact that any stationary measure for an IFS can be realised as a stationary measure for an iterate of an IFS, that after relabelling digits there is no loss of generality in assuming that there exists $a_{1},a_{2}\in \A$ for which the following property holds: For any $a\in \A$ either $$\varphi_{a}(X)\cap \varphi_{a_1}(X)=\emptyset$$ or 
$$\varphi_{a}(X)\cap \varphi_{a_2}(X)=\emptyset$$
We are now in a position to define our subsets of $\A$ so that we can perform our disintegration. For each $a\in \A$ let $$\A_{a}=\{a,a_i\}$$ where $a_i\in \{a_1,a_2\}$ has been chosen so that 
\begin{equation}
	\label{eq:empty intersection}
\varphi_{a}(X)\cap \varphi_{a_i}(X)=\emptyset.
\end{equation} We will show that Theorem \ref{thm:self-conformal} holds for the $\Omega$, $\P$, and $\{\mu_{\omega}\}_{\omega\in \Omega}$ corresponding to $\{\A_{a}\}_{a\in \A}.$ We emphasise that in this context $\Omega=\A^{\N}$. At this point it is instructive to highlight an ambiguity in our notation. Elements of $\A$ are being used to encode two objects. They encode contractions in our IFS as is standard. They also encode the sets with which we perform our disintegration, i.e. the sets $A_1,\ldots, A_k$ (using the terminology of Section \ref{sec:preliminaries}). Thus to help with our exposition, when an element of $\A$ is being used to encode a contraction we will denote it by $a$, and when an element of $\A$ is being used to encode a subset of $\A$ we will denote it by $b$. Elements of $\Omega$ will be denoted by $(b_n)_{n=1}^{\infty}.$

We now make several straightforward observations that follow from the construction of the sets $\{\A_{b}\}_{b\in \A}$. There exists $c_{1}>0$ such that for any $(b_n)\in \Omega,$ if $a,a'\in \A_{b_1}$ are distinct then 
\begin{equation}
	\label{eq:level 1 cylinder separation}
d(\varphi_{a}(X),\varphi_{a'}(X))\geq c_{1}.
\end{equation} This equation has the important consequence that there exists constants $c_{2},c_{3}>0$ such that for all $\omega\in \Omega$ we have 
\begin{equation}
	\label{eq:order 1 diameter}
c_{2}\leq Diam(X_{\omega})\leq c_{3}.
\end{equation}Combining \eqref{eq:order 1 diameter} with Lemma \ref{Lem:Sascha lemma} implies that for any $\omega\in \Omega$ and $\a\in \A^{*}$ we have 
\begin{equation}
	\label{eq:random diameters}
	Diam(\varphi_{\a}(X_{\omega}))\asymp Diam(\varphi_{\a}(X)).
\end{equation}
It also follows from the definition of $\mu_{\omega}$ and \eqref{eq:empty intersection} that for any $(b_n)\in \Omega$ and $(a_1,\ldots a_n)\in \prod_{j=1}^{n}\A_{b_j}$ we have
\begin{equation}
\label{eq:Measure of cylinders}
\mu\left(\varphi_{a_1,\ldots, a_n}(X_{\sigma^{n}\omega})\right)=\prod_{j=1}^{n}q_{a_j}^{b_j}.
\end{equation} Moreover, since each $\A_{b}$ contains at least two elements and our probability vector $\p$ is strictly positive, it follows that $$q_{\min}:=\min_{b\in \A}\min_{a\in \A_{b}}\{q_{a}^{b}\}\quad\textrm{ and }\quad  q_{max}:=\max_{b\in \A}\max_{a\in \A_{b}}\{q_{a}^{b}\} $$ satisfy 
\begin{equation}
\label{eq:q bounds}
0<q_{\min}\leq q_{max}<1.
\end{equation}

We will also require the following lemmas.

\begin{lemma}
	\label{lem:separation lemma}
	For all $(b_n)\in \Omega$ and $\a,\a'\in \cup_{n=1}^{\infty}\prod_{j=1}^{n}\A_{b_{j}}$ such that $\a$ is not a prefix of $\a'$ and $\a'$ is not a prefix of $\a,$ we have $$d(\varphi_{\a}(X),\varphi_{\a'}(X))\asymp  Diam(\varphi_{|\a\wedge \a'|}(X)).$$
\end{lemma}
\begin{proof}
Since $\varphi_{\a}(X),\varphi_{\a'}(X)\subset \varphi_{|\a\wedge \a'|}(X)$ for all $\a,\a'\in \A^*$ we trivially have 
\begin{equation}
	\label{eq:easy direction}
	d(\varphi_{\a}(X),\varphi_{\a'}(X))\leq  Diam(\varphi_{|\a\wedge \a'|}(X)).
\end{equation}
We now focus on proving an inequality in the opposite direction. Let $\a=(a_n),\a'=(a_{n}')\in \cup_{n=1}^{\infty}\prod_{j=1}^{n}\A_{b_j}$ satisfy the assumptions of the lemma. Let us also assume that they have a common prefix, i.e. $\a\wedge \a'$ exists. The case where $\a\wedge \a'$ is the empty word follows immediately from \eqref{eq:level 1 cylinder separation}. Let $x_{\a},x_{\a'}$ be such that $d(\varphi_{\a}(X),\varphi_{\a'}(X))=\|x_{\a}-x_{\a'}\|.$ These points have to exist by compactness. There exists $\tilde{x}_{\a}\in \varphi_{a_{|\a\wedge \a'|+1}}(X)$ and $\tilde{x}_{\a'}\in \varphi_{a_{|\a\wedge \a'|+1}'}(X)$ such that $\varphi_{\a\wedge \a'}(\tilde{x}_{\a})=x_{\a}$ and $\varphi_{\a\wedge \a'}(\tilde{x}_{\a'})=x_{\a'}.$ Since $a_{|\a\wedge \a'|+1}\neq a_{|\a\wedge \a'|+1}',$ it follows from \eqref{eq:level 1 cylinder separation} that $\|\tilde{x}_{\a}-\tilde{x}_{\a'}\|\geq c_{1}$. Using this inequality together with Lemma \ref{Lem:Sascha lemma} yields
\begin{equation}
	\label{eq:second direction}	
	d(\varphi_{\a}(X),\varphi_{\a'}(X))=\|x_{\a}-x_{\a'}\|=\|\varphi_{\a\wedge \a'}(\tilde{x}_{\a})-\varphi_{\a\wedge \a'}(\tilde{x}_{\a'})\|\gg Diam(\varphi_{|\a\wedge \a'|}(X)).
	\end{equation}
 \eqref{eq:easy direction} and \eqref{eq:second direction} together yield our result.
\end{proof}

\begin{lemma}
	\label{lem:Approximation lemma}
There exists $M\in \mathbb{N}$ such that for any $\omega\in \Omega,$ $r\in (0,c_{1}/2)$ and $x\in X_{\omega}$ there exists $(a_n)\in \Sigma_{\omega}$ and $N_1,N_{2}\in\mathbb{N}$ satisfying:
	\begin{enumerate}
		\item $0<N_{2}-N_{1}\leq M,$
		\item  $Diam(\varphi_{a_1\ldots a_{N_1}}(X_{\sigma^{N_1}\omega}))\asymp Diam(\varphi_{a_1\ldots a_{N_2}}(X_{\sigma^{N_2}\omega}))\asymp r,$
		\item $B(x,2r)\cap X_{\omega}\subset \varphi_{a_1\ldots a_{N_1}}(X_{\sigma^{N_{1}}\omega}),$ 
		\item $\varphi_{a_1\ldots a_{N_2}}(X_{\sigma^{N_{2}}\omega})\subset B(x,r).$
	\end{enumerate}   
\end{lemma} 
\begin{proof}
	Let $\omega=(b_n)$, $r$ and $x$ be as in the statement of our lemma. Since $x\in X_{\omega}$ there exists $(a_n)\in \Sigma_{\omega}$ such that $\Pi_{\omega}((a_n))=x$. For the rest of our proof this $(a_n)$ is fixed. We let
	$$N_{1}=\max\{n\in \N: B(x,2r)\cap X_{\omega}\subsetneq \varphi_{a_1\ldots a_{n}}(X_{\sigma^{n}\omega})\}$$ and $$N_{2}=\min\{n\in \N: \varphi_{a_1\ldots a_{n}}(X_{\sigma^{n}\omega})\subset B(x,r) \}.$$ It follows from \eqref{eq:set dynamical self-similarity}, \eqref{eq:level 1 cylinder separation} and our assumption $r\in (0,c_1/2)$ that $N_{1}$ is well defined. It is a consequence of \eqref{eq:set dynamical self-similarity} that $N_{1}<N_{2}$. The third and fourth properties in the statement of our lemma follow immediately for this choice of $N_1$ and $N_2$. It remains to shown that $N_{2}-N_{1}$ can be uniformly bounded from above and that the second statement holds. 
	
	We focus on proving the second property holds, namely:  \begin{equation}
		\label{eq:Cylinder diameter}	
		Diam(\varphi_{a_1\ldots a_{N_1}}(X_{\sigma^{N_1}\omega}))\asymp Diam(\varphi_{a_1\ldots a_{N_2}}(X_{\sigma^{N_2}\omega}))\asymp r.
	\end{equation} Since $\varphi_{a_1\ldots a_{N_2}}(X_{\sigma^{N_2}\omega})\subset \varphi_{a_1\ldots a_{N_1}}(X_{\sigma^{N_1}\omega})$ by \eqref{eq:set dynamical self-similarity}, it will suffice to show that 
	\begin{equation}
		\label{eq:WTS}
		Diam(\varphi_{a_1\ldots a_{N_2}}(X_{\sigma^{N_2}\omega}))\gg r\quad \textrm{ and } \quad Diam(\varphi_{a_1\ldots a_{N_1}}(X_{\sigma^{N_1}\omega}))\ll r.
	\end{equation}
	Let $N_{3}=\min\{n\in \N: \varphi_{a_1\ldots a_{n}}(X)\subset B(x,r) \}.$ Since $X_{\omega}\subset X$ for any $\omega\in \Omega$ it follows from the definition that $N_{2}\geq N_{3}.$ Appealing to well known properties of self-conformal IFSs, it can be shown that the sequence $(Diam(\varphi_{a_1\ldots a_{n}}(X)))_{n=1}^{\infty}$ satisfies $$Diam(\varphi_{a_1\ldots a_{n+1}}(X))\geq \kappa Diam(\varphi_{a_1\ldots a_{n}}(X))$$ for all $n\in \N$ for some $\kappa\in (0,1)$ that doesn't depend upon $(a_n)$. Using this property together with the fact $\Pi_{\omega}((a_n))=x,$ it follows that  $Diam(\varphi_{a_1\ldots a_{N_{3}}}(X))\gg r$. Combining this with \eqref{eq:random diameters} then implies $$Diam(\varphi_{a_1\ldots a_{N_{3}}}(X_{\sigma^{N_{3}}\omega}))\gg r.$$ $N_{2}\leq N_{3}$ so $\varphi_{a_1\ldots a_{N_{3}}}(X_{\sigma^{N_{3}}\omega})\subset \varphi_{a_1\ldots a_{N_{2}}}(X_{\sigma^{N_{2}}\omega})$ by \eqref{eq:set dynamical self-similarity}. Therefore $Diam(\varphi_{a_1\ldots a_{N_{2}}}(X_{\sigma^{N_{2}}\omega}))\geq Diam(\varphi_{a_1\ldots a_{N_{3}}}(X_{\sigma^{N_{3}}\omega})).$ Thus the above implies that
	\begin{equation}
		\label{eq:first part}
			Diam(\varphi_{a_1\ldots a_{N_{2}}}(X_{\sigma^{N_{2}}\omega}))\gg r
	\end{equation}and we have proved the first part of \eqref{eq:WTS}. 
	
	We now focus on the second inequality in \eqref{eq:WTS}. It follows from the definition of $N_{1}$ and \eqref{eq:set dynamical self-similarity} that there exists distinct $a,a'\in \A_{b_{N_1+1}}$ such that $$B(x,2r)\cap \varphi_{a_1\ldots a_{N_{1}}a}(X_{\sigma^{N_{1}+1}\omega})\neq\emptyset \quad \textrm{ and }\quad B(x,2r)\cap \varphi_{a_1\ldots a_{N_{1}}a'}(X_{\sigma^{N_{1}+1}\omega})\neq\emptyset.$$ Therefore $$d(\varphi_{a_1\ldots a_{N_{1}}a}(X_{\sigma^{N_{1}+1}\omega}),\varphi_{a_1\ldots a_{N_{1}}a'}(X_{\sigma^{N_{1}+1}\omega}))\ll r.$$ Since $X_{\omega}\subset X$ for all $\omega\in \Omega$ this now implies
	$$d(\varphi_{a_1\ldots a_{N_{1}}a}(X),\varphi_{a_1\ldots a_{N_{1}}a'}(X))\ll r.$$ By Lemma \ref{lem:separation lemma} this implies $$Diam(\varphi_{a_1\ldots a_{N_{1}}}(X)))\ll r.$$ Since $\varphi_{a_1\ldots a_{N_{1}}}(X_{\sigma^{N_{1}}\omega}))\subset \varphi_{a_1\ldots a_{N_{1}}}(X)$ for all $\omega\in \Omega,$ this in turn implies
	\begin{equation}
		\label{eq:second part}
		Diam(\varphi_{a_1\ldots a_{N_{1}}}(X_{\sigma^{N_{1}}\omega})))\ll r.
	\end{equation}
Thus we have established the second part of \eqref{eq:WTS}. As previously remarked, this together with \eqref{eq:first part} implies \eqref{eq:Cylinder diameter}. Thus the second property in the statement of our lemma holds. 
	
We now focus on establishing the uniform upper bound for $N_2 - N_1$.  By \eqref{eq:random diameters} we have 
 $$Diam(\varphi_{a_1\ldots a_{N_1}}(X_{\sigma^{N_1}\omega}))\asymp Diam(\varphi_{a_1\ldots a_{N_1}}(X))$$
 and $$Diam(\varphi_{a_1\ldots a_{N_2}}(X_{\sigma^{N_2}\omega}))\asymp Diam(\varphi_{a_1\ldots a_{N_2}}(X)).$$ Combining these equations with our second property and Lemma \ref{Lem:Sascha lemma} we have
 $$1\asymp \frac{r}{r}\asymp \frac{Diam(\varphi_{a_1\ldots a_{N_2}}(X_{\sigma^{N_2}\omega}))}{Diam(\varphi_{a_1\ldots a_{N_1}}(X_{\sigma^{N_1}\omega}))}\asymp\frac{Diam(\varphi_{a_1\ldots a_{N_2}}(X))}{Diam(\varphi_{a_1\ldots a_{N_1}}(X))}\asymp Diam(\varphi_{a_{N_1+1}\ldots a_{N_2}}(X)).$$
 Since $Diam(\varphi_{a_{N_1+1}\ldots a_{N_2}}(X))\ll \gamma^{N_2-N_1}$ for some $\gamma\in (0,1)$ it follows that $1\ll \gamma^{N_2-N_1}$. Hence $N_{2}-N_{1}$ must be uniformly bounded from above.
\end{proof}

Now we give our proof of Theorem \ref{thm:self-conformal}.

\begin{proof}[Proof of Theorem \ref{thm:self-conformal}]	Statement $1$. of this theorem is the content of Proposition \ref{prop:Disintegration prop}. Let us now focus on statement $2$. Let $\omega=(b_n)\in \Omega,$ $x\in \supp(\mu_{\omega})$ and $0<r\leq c_{1}/2$. We emphasise that to establish statement $2$ it suffices to consider $r$ in this interval. Applying Lemma \ref{lem:Approximation lemma} we can assert that there exists two words $\a,\a'$ such that $\a$ is a prefix of $\a',$ $|\a'|-|\a|\ll 1,$ and $$B(x,2r)\cap X_{\omega}\subset \varphi_{\a}(X_{\sigma^{|\a|}\omega})\quad \textrm{ and }\quad\varphi_{\a'}(X_{\sigma^{|\a'|}\omega})\subset B(x,r).$$ Therefore by \eqref{eq:Measure of cylinders} we have $$\mu_{\omega}(B(x,2r))\leq \mu_{\omega}(\varphi_{\a}(X_{\sigma^{|\a|}\omega}))=\prod_{j=1}^{|\a|}q_{a_j}^{b_{j}}$$ and $$\prod_{j=1}^{|\a'|}q_{a_{j}'}^{b_{j}}= \mu_{\omega}(\varphi_{\a'}(X_{\sigma^{|\a'|}\omega}))\leq \mu_{\omega}(B(x,r)).$$ Now using these inequalities together with the facts $\a$ is a prefix of $\a'$ and $|\a'|-|\a|\ll 1$ we have 
	\begin{align*}
	\frac{\mu_{\omega}(B(x,2r))}{\mu_{\omega}(B(x,r))}\ll \frac{\prod_{j=1}^{|\a|}q_{a_{j}}^{b_{j}}}{\prod_{j=1}^{|\a'|}q_{a_{j}'}^{b_{j}}}&=\left(\prod_{j=|\a|+1}^{|\a'|}q_{a_{j}'}^{b_{j}}\right)^{-1}\\
	&\leq q_{min}^{-(|\a'|-|\a|)}\\
	&\ll 1.
	\end{align*}
	In the final line we used \eqref{eq:q bounds}. Therefore $\mu(B(x,2r))\ll \mu(B(x,r))$ and statement $2$ follows. 

Now we bring our attention to statement $3$ of our theorem. Let $(b_n)\in \Omega$, $x\in \supp(\mu_{\omega})$, $y\in\mathbb{R}^{d}$, $0<r<c_{1}$ and $0<\epsilon<1$. We emphasise that there is no loss of generality in restricting to $0<r<c_{1}$ and $0<\epsilon<1$. We may also assume without loss of generality that $y\in \supp(\mu_{\omega})$ and that $B(y,\epsilon r)\subset B(x,r).$ This final reduction is permissible because of the doubling property guaranteed by the already established statement $2$. 

By Lemma \ref{lem:Approximation lemma} there exists $\a,\a'\in \cup_{n=1}^{\infty}\prod_{j=1}^{n}\A_{b_j}$ such that:
\begin{itemize}
	\item $Diam(\varphi_{\a}(X_{\sigma^{|\a|}\omega}))\asymp r$
	\item $Diam(\varphi_{\a'}(X_{\sigma^{|\b|}\omega}))\asymp \epsilon r$
	\item $B(x,r)\cap X_{\omega}\subset \varphi_{\a}(X_{\sigma^{|\a|}\omega})$
	\item $\varphi_{\a'}(X_{\sigma^{|\a'|}\omega})\subset B(y,\epsilon r)$.
\end{itemize}
The analysis given in the proof of statement $2$ also yields the following properties of $\a$ and $\a'$:
\begin{equation}
	\label{eq:Comparable A}
	\mu_{\omega}(B(x,r))\asymp \mu_{\omega}(\varphi_{\a}(X_{\sigma^{|\a|}\omega})),
\end{equation} and
	\begin{equation}
		\label{eq:Comparable B}
	\mu_{\omega}(B(y,\epsilon r))\asymp \mu_{\omega}(\varphi_{\a'}(X_{\sigma^{|\a'|}\omega}))	
	\end{equation}
It follows from the inclusions $B(y,\epsilon r)\subset B(x,r)$ and $B(x,r)\cap X_{\omega}\subset \varphi_{\a}(X_{\sigma^{|\a|}\omega})$ that $\a$ must be a prefix of $\a'$. The following is a consequence of this fact, the above, \eqref{eq:random diameters}, and Lemma \ref{Lem:Sascha lemma}:
\begin{equation}
	\label{eq:epsilon asymptotics}
\epsilon \asymp \frac{\epsilon r}{r}\asymp  \frac{Diam(\varphi_{\a'}(X_{\sigma^{|\a'|}\omega}))}{Diam(\varphi_{\a}(X_{\sigma^{|\a|}\omega}))}\asymp \frac{Diam(\varphi_{\a'}(X))}{Diam(\varphi_{\a}(X))}\asymp Diam(\varphi_{a_{|\a|+1}'\ldots a_{|\a'|}'}(X)).
\end{equation} Appealing to well known properties of self-conformal IFSs, it can be shown that there exists $\gamma,c_1\in (0,1)$ such that $Diam(\varphi_{\a}(X))\geq c\gamma^{|\a|}$ for all $\a\in \A^{*}$. Substituting this inequality into \eqref{eq:epsilon asymptotics} implies that there exists $c_{2},c_{3}>0$ such that
\begin{equation}
	\label{eq:length and log}
|\a'|-|\a|\geq -c_{2}\log \epsilon -c_{3}.
\end{equation}
Now we observe
\begin{align*}
	\frac{\mu_{\omega}(B(y,\epsilon r)\cap B(x,r))}{\mu_{\omega}(B(x, r))}\stackrel{\eqref{eq:Comparable A},\, \eqref{eq:Comparable B}}{\asymp}  \frac{\mu_{\omega}(\varphi_{\a'}(X_{\sigma^{|\a'|}\omega}))}{\mu_{\omega}(\varphi_{\a}(X_{\sigma^{|\a|}\omega}))}
	&\stackrel{\eqref{eq:Measure of cylinders}}{=}\left(\prod_{j=1}^{|\a'|}q_{a_{j}'}^{b_{j}}\right)\cdot \left(\prod_{j=1}^{|\a|}q_{a_{j}}^{b_{j}}\right)^{-1}\\
	&=\prod_{j=|\a|+1}^{|\a'|}q_{a_{j}'}^{b_{j}}\\
	&\leq q_{max}^{|\a'|-|\a|}\\
	&\stackrel{\eqref{eq:length and log}}{\ll}  q_{max}^{-c_{2}\log \epsilon}\\
	&=\epsilon^{\alpha}
\end{align*}
where $$\alpha=-c_{2}\log q_{max}.$$ In the second equality we have used that $\a$ is a prefix of $\a'$. The fact $\alpha>0$ follows from \eqref{eq:q bounds}. Thus statement $3$ holds and our proof is complete.
\end{proof}

\subsection{Proof of Theorem \ref{thm:self-similar}}
Throughout this section we fix an affinely irreducible self-similar IFS $\Phi$ and a probability vector $\p$. We begin by making some simplifications that are permissible because of our affinely irreducible assumption.

Appealing to a compactness argument, it can be shown that there exists $\epsilon_{0}>0$ such that for any affine subspace $W<\R^{d}$ we have $X\setminus W^{(\epsilon_0)}\neq \emptyset.$ Replacing $\epsilon_0$ with a potentially smaller constant, it follows that there exists $x_1,\ldots, x_{n}\in X$ such that for the following properties are satisfied:
\begin{itemize}
	\item $B(x_i,\epsilon_0)\cap B(x_j, \epsilon_0)=\emptyset$ for $i\neq j$.
	\item For any $x\in X$ there exists $1\leq i\leq n$ such that $B(x_i,\epsilon_{0})\subset B(x,3\epsilon_{0})$.
	\item For any affine subspace $W<\mathbb{R}^{d}$ there exists $1\leq i\leq n$ such that $B(x_i,\epsilon_0)\cap W^{(\epsilon_0)}=\emptyset$.
\end{itemize}
Now using the self-similar structure of $X$, these properties imply that there exists $N\in \N$ and $\a_1,\ldots,\a_n\in \A^{N}$ such that: 
\begin{itemize}
	\item $\varphi_{\a_i}(X)\cap \varphi_{\a_j}(X)=\emptyset$ for $i\neq j$.
	\item For any $x\in X$ there exists $1\leq i\leq n$ such that $\varphi_{\a_{i}}(X)\subset B(x,3\epsilon_{0})$.
	\item For any affine subspace $W<\mathbb{R}^{d}$ there exists $1\leq i\leq n$ such that $\varphi_{\a_i}(X)\cap W^{(\epsilon_0)}=\emptyset$.
\end{itemize}
Now replacing $N$ with a larger integer if necessary, and $\epsilon_0>0$ with a potentially smaller quantity if necessary, the above implies that there exists $N\in \N$ and $\a_{1},\ldots \a_{n}$ such that the following properties are satisfied:
\begin{itemize}
	\item For any $\a\in \A^{2N}$ there exists $1\leq i\leq n$ such that
	$$\varphi_{\a}(X)\cap \bigcup_{j=1}^{n}\varphi_{\a_{i}\a_{j}}(X)=\emptyset.$$
	\item For any $1\leq i\leq n$ we have $\varphi_{\a_{i}\a_{j}}(X)\cap \varphi_{\a_{i}\a_{j'}}(X)=\emptyset$ for $j\neq j'$.
	\item For any affine subspace $W<\R^{d}$ and $1\leq i\leq n,$ there exists $1\leq j\leq n$ such that $\varphi_{\a_i\a_j}(X)\cap W^{(\epsilon_0)}=\emptyset$.
\end{itemize}
In the third property listed above, we have relied upon the fact that our contractions are similarities and therefore map affine subspaces to affine subspaces and contract distances in a uniform way.

It is clear now that after iterating our IFS and relabelling the digits, we can assume without loss of generality that there exists $\{a_{1,1},\ldots, a_{n,1}\},\ldots, \{a_{1,n},\ldots, a_{n,n}\}\subset \A$ and $\epsilon_0>0$ such that our IFS satisfies the following properties:
\begin{itemize}
	\item For any $a\in \A$ there exists $1\leq i\leq n$ such that $$\varphi_{a}(X)\cap \bigcup_{j=1}^{n}\varphi_{a_{j,i}}(X)=\emptyset.$$
	\item For any any $1\leq i\leq n$ we have $\varphi_{a_{j,i}}(X)\cap \varphi_{a_{j',i}}(X)=\emptyset$ for $j\neq j'$.
	\item For any affine subspace $W<\R^{d}$ and $1\leq i\leq n,$ there exists $1\leq j\leq n$ such that $\varphi_{a_{j,i}}(X)\cap W^{(\epsilon_0)}=\emptyset$.
\end{itemize}
For each $a\in \A$ we let $$\A_{a}=\{a,a_{1,i},\ldots a_{n,i}\}$$ where $1\leq i\leq n$ has been chosen so that 
$$\varphi_{a}(X)\cap \bigcup_{j=1}^{n}\varphi_{a_{j,i}}(X)=\emptyset.$$ 
 As in the proof of Theorem \ref{thm:self-conformal} an ambiguity in our notation arises. Elements of $\A$ encode contractions in our IFS and the sets with which we perform our disintegration. To help with our exposition, when an element of $\A$ is being used to encode a contraction we will denote it by $a$, and when an element of $\A$ is being used to encode a subset of $\A$ we will denote it by $b$.

It follows from the above that there exists $c_{1}>0$ such that for any $b\in \A,$ if $a,a'\in \A_{b}$ are distinct then 
\begin{equation}
	\label{eq:level 1 cylinder separationA}
	d(\varphi_{a}(X),\varphi_{a'}(X))\geq c_{1}
\end{equation} 
and
\begin{equation}
	\label{eq:emptyintersection B}
	\varphi_{a}(X)\cap\varphi_{a'}(X)=\emptyset.
\end{equation}
It follows now from the definition of $\mu_{\omega}$ and \eqref{eq:emptyintersection B} that for any $(b_n)\in \Omega$ and $(a_1,\ldots a_n)\in \prod_{j=1}^{n}\A_{b_j},$ we have
\begin{equation}
	\label{eq:Measure of cylindersA}
	\mu\left(\varphi_{a_1,\ldots, a_n}(X_{\sigma^{n}\omega})\right)=\prod_{j=1}^{n}q_{a_j}^{b_j}.
\end{equation} As in the proof of Theorem \ref{thm:self-conformal}, since each $\A_{b}$ contains at least two elements and our probability vector $\p$ is strictly positive, it follows that $$q_{\min}:=\min_{b\in \A}\min_{a\in \A_{b}}\{q_{a}^{b}\}\quad\textrm{ and }\quad  q_{max}:=\max_{b\in \A}\max_{a\in \A_{b}}\{q_{a}^{b}\} $$ satisfy 
\begin{equation}
	\label{eq:q boundsA}
	0<q_{\min}\leq q_{max}<1.
\end{equation}

 We now show that Theorem \ref{thm:self-similar} holds for the $\Omega$, $\P$ and $\{\mu_{\omega}\}_{\omega}$ corresponding to this choice of $\{\A_{b}\}_{b\in \A}.$ The first step is the following proposition which uniformly bounds how much mass a $\mu_{\omega}$ can give to a neighbourhood of an affine subspace.

\begin{prop}
	\label{prop:planar decay}
	There exists $C,\alpha>0$ such that for any $\omega\in \Omega,$ $W<\R^{d}$ and $\epsilon>0$ we have $\mu_{\omega}(W^{(\epsilon)})\leq C\epsilon^{\alpha}.$
\end{prop}
\begin{proof}
We begin by remarking that to prove our proposition it suffices to consider $\epsilon< \epsilon_{0}$ where $\epsilon_{0}$ is as above. Let $$r_{min}:=\min_{a\in \A}\min_{x,y\in\mathbb{R}^{d},\, x\neq y}\left\{\frac{\|\varphi_{a}(x)-\varphi_{a}(y)\|}{\|x-y\|}\right\}.$$ Since our IFS consists of similarities we must have $r_{min}>0$. We also remark that any distinct $x,y\in \R^{d}$ can be taken in the second minimum in the definition of $r_{min}$.  Our proof will rely upon showing that  for any $\omega\in \Omega$, $\epsilon<\epsilon_{0}$ and $W<\R^{d}$ there exists an affine subspace $W_{1}<\mathbb{R}^{d}$ such that 
	\begin{equation}
		\label{eq:planar decay}
	\mu_{\omega}(W^{(\epsilon)})\leq (1-q_{min})\mu_{\sigma\omega}(W_{1}^{(\epsilon \cdot r_{min}^{-1})}).
	\end{equation}
As such, let $\omega=(b_n)\in \Omega,$ $W<\mathbb{R}^{d}$ and $\epsilon<\epsilon_{0}$. Then by \eqref{eq:dynamical self-similarity} we have 
	\begin{align*}
		\mu_{\omega}(W^{(\epsilon)})=\sum_{a\in \A_{b_1}}q^{b_1}_{a}\varphi_{a}\mu_{\sigma \omega}(W^{(\epsilon)}).
	\end{align*} By construction of the sets $\{\A_{b}\}_{b\in \A}$ there must exist $a\in \A_{b_1}$ such that $W^{(\epsilon)}\cap \varphi_{a}(X)=\emptyset.$ Since $\supp( \varphi_{a}\mu_{\sigma \omega})\subset \varphi_{a}(X)$ it follows that $\varphi_{a}\mu_{\sigma \omega}(W^{(\epsilon)})=0$ for some $a\in \A_{b_{1}}$. Consequently, if we choose $a^{*}\in \A_{b_{1}}$ for which  $\varphi_{a^*}\mu_{\sigma \omega}(W^{(\epsilon)})=\max_{a\in \A_{b_1}}\{\varphi_{a}\mu_{\sigma \omega}(W^{(\epsilon)})\},$  we have 
	\begin{equation}
		\label{eq:shaving measure}\mu_{\omega}(W^{(\epsilon)})\leq (1-q_{min})\varphi_{a^*}\mu_{\sigma \omega}(W^{(\epsilon)}).
	\end{equation}
	 Finally, using the fact that our IFS consists of similarities, we have $$\varphi_{a^*}^{-1}(W^{(\epsilon)})\subset W_{1}^{(\epsilon\cdot r_{\min}^{-1})}$$ for some $W_{1}<\R^d$. Therefore 
	 \begin{equation}
	 	\label{eq:inclusion measure bound} \varphi_{a^*}\mu_{\sigma \omega}(W^{\epsilon})\leq \mu_{\sigma \omega}(W_{1}^{\epsilon\cdot r_{min}^{-1}}).
	 \end{equation}
	 Combining \eqref{eq:shaving measure} and \eqref{eq:inclusion measure bound} implies \eqref{eq:planar decay}.
	
	Equipped with \eqref{eq:planar decay} we can now finish our proof. Let $\omega\in \Omega$, $W<\R^{d}$ and $\epsilon<\epsilon_{0}$ be arbitrary. Let $M\in \mathbb{N}$ be the unique integer satisfying 
	\begin{equation}
		\label{eq:M cut off}
		\epsilon_{0}r_{min}^{M}\leq \epsilon\leq \epsilon_{0}r_{min}^{M-1}.
	\end{equation}
	 Repeatedly applying \eqref{eq:planar decay} yields a sequence of affine subspaces $W_{1},W_{2},\ldots,W_{M}$ such that $$\mu_{\omega}(W^{(\epsilon)})\leq (1-q_{min})\mu_{\sigma \omega}(W_{1}^{(\epsilon r_{\min}^{-1})})\leq \cdots \leq (1-q_{min})^{M}\mu_{\sigma^{M} \omega}(W_{M}^{(\epsilon r_{\min}^{-M})})\leq (1-q_{min})^{M}.$$ Thus $\mu_{\omega}(W^{(\epsilon)})\leq (1-q_{min})^{M}$. Combining this inequality with \eqref{eq:M cut off} yields 
	 $$\mu_{\omega}(W^{(\epsilon)})\ll \epsilon^{\alpha}$$ for $$\alpha=\frac{\log (1-q_{min})}{\log r_{\min}}.$$ Thus our result holds. 
\end{proof}
Duplicating the arguments given in the proof of Lemma \ref{lem:Approximation lemma} it is possible to show that the following analogous statement holds for the sets $\{\A_{b}\}_{b\in \A}$ defined in this section. In this lemma $c_{1}$ is as in \eqref{eq:level 1 cylinder separationA}.

\begin{lemma}
	\label{lem:Approximation lemma2}
	There exists $M\in \mathbb{N}$ such that for any $\omega\in \Omega,$ $r\in (0,c_{1}/2)$ and $x\in X_{\omega}$ there exists $(a_n)\in \Sigma_{\omega}$ and $N_1,N_{2}\in\mathbb{N}$ satisfying:
	\begin{enumerate}
		\item $0<N_{2}-N_{1}\leq M,$
		\item  $Diam(\varphi_{a_1\ldots a_{N_1}}(X_{\sigma^{N_1}\omega}))\asymp Diam(\varphi_{a_1\ldots a_{N_2}}(X_{\sigma^{N_2}\omega}))\asymp r,$
		\item $B(x,2r)\cap X_{\omega}\subset \varphi_{a_1\ldots a_{N_1}}(X_{\sigma^{N_{1}}\omega}),$ 
		\item $\varphi_{a_1\ldots a_{N_2}}(X_{\sigma^{N_2}\omega})\subset B(x,r).$
	\end{enumerate}   
\end{lemma} 
Equipped with Proposition \ref{prop:planar decay} and Lemma \ref{lem:Approximation lemma2} we can now prove Theorem \ref{thm:self-similar}.

\begin{proof}[Proof of Theorem \ref{thm:self-similar}]
Statement $1$ of this theorem holds because of Proposition \ref{prop:Disintegration prop}. Statement $2$ follows from an analogous argument to that used in the proof of statement $2$ from Theorem \ref{thm:self-conformal} where we appeal to  \eqref{eq:Measure of cylindersA} and Lemma  \ref{lem:Approximation lemma2} instead of \eqref{eq:Measure of cylinders} and Lemma \ref{lem:Approximation lemma}. As such we omit this argument. We now focus on statement $3$. Fix $\omega=(b_n)\in \Omega$, $x\in \supp(\mu_{\omega}),$ an affine subspace $W<\mathbb{R}^{d},$ $r\in (0,c_1/2)$ and $\epsilon\in (0,1)$. It suffices to consider $r$ and $\epsilon$ contained in this restricted domain. 

It follows from an application of Lemma \ref{lem:Approximation lemma2}, and the arguments used in the proof of Theorem \ref{thm:self-conformal}, that there exists $D>0$ depending only on our IFS and $\a\in \cup_{n=1}^{\infty}\prod_{j=1}^{n}\A_{b_j}$ such that 
\begin{equation}
	\label{eq:cylinder cover ball}
B(x,r)\cap X_{\omega}\subset \varphi_{\a}(X_{\sigma^{|\a|}\omega}),
\end{equation}
\begin{equation}
	\label{eq:self-similar diameter}
	\prod_{j=1}^{|\a|}r_{a_j}\geq \frac{r}{D}
\end{equation}
 and 
 \begin{equation}
 	\label{eq:ball to cylinder}
 	\mu_{\omega}(B(x,r))\asymp \mu_{\omega}(\varphi_{\a}(X_{\sigma^{|\a|}\omega})).
 \end{equation} 
In \eqref{eq:self-similar diameter} we have adopted the standard notation that for $a\in \A$ we let $r_{a}\in(0,1)$ be such that $\|\varphi_{a}(x)-\varphi_{a}(y)\|=r_{a}\|x-y\|$ for all $x,y\in \R^{d}$.

We now observe the following: 
\begin{align*}
	\mu_{\omega}(W^{(\epsilon r)}\cap B(x,r))&\stackrel{\eqref{eq:dynamical self-similarity}}{=}\sum_{\a'\in \prod_{j=1}^{|\a|}\A_{b_{j}}}\prod_{j=1}^{|\a|}q_{a_{j}}^{b_{j}}\cdot \varphi_{\a'}\mu_{\sigma^{|\a|}\omega}(W^{(\epsilon r)}\cap B(x,r))\\
	&\stackrel{\eqref{eq:cylinder cover ball}}{=}\prod_{j=1}^{|\a|}q_{a_{j}}^{b_{j}}\cdot \varphi_{\a}\mu_{\sigma^{|\a|}\omega}(W^{(\epsilon r)}\cap B(x,r))\\
	&\leq \prod_{j=1}^{|\a|}q_{a_{j}}^{b_{j}}\cdot \varphi_{\a}\mu_{\sigma^{|\a|}\omega}(W^{(\epsilon r)})\\
	&\leq \prod_{j=1}^{|\a|}q_{a_{j}}^{b_{j}}\cdot \mu_{\sigma^{|\a|}\omega}(W_{1}^{(D\epsilon )}).
\end{align*}
Here $W_{1}=\varphi_{\a}^{-1}(W)$ and in the final line we have used \eqref{eq:self-similar diameter}. To conclude the second equality in the above we also used that $\varphi_{\a'}(X_{\sigma^{|\a|}\omega})\cap \varphi_{\a'}(X_{\sigma^{|\a|}\omega})=\emptyset$ for distinct $\a',\a''\in \prod_{j=1}^{|\a|}\A_{b_{j}}.$ Now applying Proposition \ref{prop:planar decay}, \eqref{eq:Measure of cylindersA} and \eqref{eq:ball to cylinder} we have 
\begin{align*}
	\mu_{\omega}(W^{(\epsilon r)}\cap B(x,r))&\ll \prod_{j=1}^{|\a|}q_{a_{j}}^{b_{j}}\cdot  \epsilon^{\alpha}\\
	&=\epsilon^{\alpha}\mu_{\omega}(\varphi_{\a}(X_{\sigma^{|\a|}\omega}))\\
	&\ll \epsilon^{\alpha}\mu_{\omega}(B(x,r)).
\end{align*}
Thus statement $3$ holds and our proof is complete.

\end{proof}
\section{An example}
\label{sec:example}
In this section we give an example of a self-similar measure coming from an IFS that satisfies the open set condition, for which there exists no $C,\alpha>0$ such for any $x\in \supp(\mu)$, affine subspace $W<\R^{d}$, $0< r\leq 1$ and $\epsilon>0$ we have $$\mu\left(W^{(\epsilon r)}\cap B(x,r)\right)\leq C\epsilon^{\alpha}\mu(B(x,r)).$$ We do not know whether this example is well know, but as several colleagues have asked us to explain it to them we have decided to include it here.

\begin{example}
Let $\Phi=\{\varphi_{0},\varphi_{1}\}$ be the self-similar IFS acting on $\R$ where $\varphi_{0}(x)=\frac{x}{2}$ and $\varphi_{1}(x)=\frac{x+1}{2}$. It is well known that this IFS satisfies the open set condition. Let us now pick $p_{0}\in(0,1/2)$ and let $p_{1}=1-p_0$. We let $\mu$ be self-similar measure corresponding to $(p_0,p_1)$. In what follows we will use the well known fact that for any $(a_1,\ldots, a_n)\in \cup_{n=1}^{\infty}\{0,1\}^{n}$ the measure $\mu$ satisfies 
\begin{equation}
	\label{eq:Base 2 mass}
\mu\left(\left[\sum_{j=1}^{n}\frac{a_j}{2^{j}},\sum_{j=1}^{n}\frac{a_j}{2^j}+\frac{1}{2^n}\right]\right)=p_{0}^{\#\{1\leq j\leq n:a_{j}=0\}}p_{1}^{\#\{1\leq j\leq n:a_{j}=1\}}.
\end{equation}For each $n\in \mathbb{N}$ consider the interval $$\left(\frac{1}{2}-\frac{1}{2^{\lfloor \frac{n \log p_0}{\log p_1}\rfloor}},\frac{1}{2}+\frac{1}{2^{n}}\right).$$ By \eqref{eq:Base 2 mass} we have 
\begin{equation}
	\label{eq:measure of big ball} 
	\mu\left(\left(\frac{1}{2}-\frac{1}{2^{\lfloor \frac{n \log p_0}{\log p_1}\rfloor }},\frac{1}{2}+\frac{1}{2^{n}}\right)\right)=p_{0}p_{1}^{\lfloor \frac{n \log p_0}{\log p_{1}}\rfloor-1}+p_{1}p_{0}^{n-1}\asymp p_{0}^{n}
\end{equation}
and 
\begin{equation}
	\label{eq:measure of small ball} 
	\mu\left(\left(\frac{1}{2}-\frac{1}{2^{\lfloor \frac{n \log p_0}{\log p_1}\rfloor }},\frac{1}{2}\right)\right)=p_0p_{1}^{\lfloor \frac{n \log p_0}{\log p_1}\rfloor-1}\asymp p_{0}^{n}.
\end{equation} For each $n\in \N$ let $$x_{n}=\frac{2^{-1}-2^{-\lfloor \frac{n \log p_0}{\log p_1}\rfloor}+2^{-1}+2^{-n}}{2},\, y_{n}=\frac{2^{-1}-2^{-\lfloor \frac{n \log p_0}{\log p_1}\rfloor}+2^{-1}}{2},\, r_{n}=\frac{1}{2}\left(2^{-n}+2^{-\lfloor \frac{n \log p_0}{\log p_1}\rfloor }\right) $$ and $$\epsilon_{n}=\frac{2^{-\lfloor \frac{n \log p_0}{\log p_1}\rfloor}}{2^{-n}+2^{-\lfloor \frac{n \log p_0}{\log p_1}\rfloor} }.$$ Combining \eqref{eq:measure of big ball} and \eqref{eq:measure of small ball}, and using the notation above, we have 
\begin{equation}
	\label{eq:final example line}
\mu(B(y_n,\epsilon_{n}r_{n})\cap B(x_n,r_n))\asymp \mu(B(x_n,r_n))
\end{equation}
for all $n\in \N$. However, since $p_0\in (0,1/2)$ we know that $\frac{\log p_0}{\log p_1}>1$ and therefore $\epsilon_{n}\to 0$ as $n\to\infty$. Therefore \eqref{eq:final example line} implies that there cannot exist $C,\alpha>0$ such that for any $x\in \supp(\mu)$, $y\in \R$, $0< r\leq 1$ and $\epsilon>0$ we have $$\mu\left(B(y,\epsilon r)\cap B(x,r)\right)\leq C\epsilon^{\alpha}\mu(B(x,r)).$$
\end{example}

\noindent \textbf{Acknowledgements.} The author was supported by an EPSRC New Investigators Award
(EP/W003880/1).

\end{document}